\documentclass[12pt]{amsart}
\usepackage{a4}
\usepackage{latexsym}
\usepackage{amsfonts}
\usepackage{amssymb}
\usepackage{amsmath}
\usepackage{amsthm}
\newtheorem{theorem}{Theorem}[section]
\newtheorem{lemma}[theorem]{Lemma}
\newtheorem{corollary}[theorem]{Corollary}

\title[Convexity and a sum-product type estimate]
{Convexity and a sum-product type estimate}
\author{Liangpan Li and Oliver Roche-Newton}

\address{Department of Mathematical Sciences,
Loughborough University, LE11 3TU, UK}
 \email{liliangpan@gmail.com}

\address{Department of Mathematics, University Walk, Bristol, BS8 1TW, UK}
 \email{maorn@bristol.ac.uk}

\subjclass[2000]{11B75}

\keywords{sumset, difference set, productset, convexity, additive
energy}

\date{}

\begin{document}

\begin{abstract}
In this paper we further study the relationship between convexity
and additive growth, building on the work of Schoen and Shkredov
(\cite{SS}) to get some improvements to earlier results of Elekes,
Nathanson and Ruzsa (\cite{ENR}). In particular, we show that for
any finite set $A\subset{\mathbb{R}}$ and any strictly convex or
concave function $f$,
\[|A+f(A)|\gg{\frac{|A|^{24/19}}{(\log|A|)^{2/19}}}\] and
\[\max\{|A-A|,\ |f(A)+f(A)|\}\gg{\frac{|A|^{14/11}}{(\log|A|)^{2/11}}}.\]
For the latter of these inequalities, we go on to consider the
consequences for a sum-product type problem.

\end{abstract}
\maketitle
\begin{section}{Introduction}
Given a finite set $A\subset{\mathbb{R}}$, the elements of $A$ can be labeled in ascending order, so that
$$a_1<a_2<\cdots<a_n.$$
$A$ is said to be \textit{convex}, if
$$a_i-a_{i-1}<a_{i+1}-a_i,$$
for all $2\leq{i}\leq{n-1}$, and it was proved by Elekes, Nathanson
and Ruzsa (\cite{ENR}) that $|A\pm{A}|\geq{|A|^{3/2}}$, an estimate
which stood as the best known for a decade, under various guises.
Schoen and Shkredov (\cite{SS}) recently made significant progress
by proving that for any convex set $A$,
$$|A-A|\gg{\frac{|A|^{8/5}}{(\log|A|)^{2/5}}},$$
and
$$|A+A|\gg{\frac{|A|^{14/9}}{(\log|A|)^{2/3}}}.$$
See \cite{SS} and the references contained within for more details on this problem and its history.

In \cite{ENR}, a number of other results were proved connecting
convexity with large sumsets. In particular, it was shown that, for
any convex or concave function $f$ and any finite set
$A\subset{\mathbb{R}}$,
\begin{equation}
\max\{|A+A|,|f(A)+f(A)|\}\gg{|A|^{5/4}}, \label{Aorf(A)}
\end{equation}
and
\begin{equation}
|A+f(A)|\gg{|A|^{5/4}}.
\label{A+f(A)}
\end{equation}
By choosing particularly interesting convex or concave functions
$f$, these results immediately yield interesting corollaries. For
example, if we choose $f(x)=\log x$, then (\ref{Aorf(A)})
immediately yields a sum-product estimate. Furthermore, if
$f(x)=1/x$, then (\ref{A+f(A)}) gives information about another
problem posed by Erd\H{o}s and Szemer\'{e}di (\cite{ES}) .

In this paper, the methods used by Schoen and Shkredov (\cite{SS}) are developed further in order to
improve on some other results from \cite{ENR}. In particular,
the bounds in (\ref{Aorf(A)}) and (\ref{A+f(A)}) are improved slightly,
in the form of the following results.
\begin{theorem}\label{theorem:main} Let $f$ be any continuous, strictly convex
 or concave function on the reals, and $A,C\subset\mathbb{R}$ be any finite sets
 such that $|A|\approx|C|$. Then
$$|f(A)+C|^6|A-A|^5\gg{\frac{|A|^{14}}{(\log|A|)^2}}.$$
In particular, choosing $C=f(A)$, this implies that
$$\max\{|f(A)+f(A)|,|A-A|\}\gg{\frac{|A|^{14/11}}{(\log|A|)^{2/11}}}.$$
\end{theorem}
\begin{theorem}\label{theorem:main2}Let $f$ be any continuous, strictly convex
 or concave function on the reals, and $A,C\subset\mathbb{R}$ be any finite sets
 such that $|A|\approx|C|$. Then
$$|f(A)+C|^{10}|A+A|^9\gg{\frac{|A|^{24}}{(\log|A|)^2}}.$$
In particular, choosing $C=f(A)$, this implies that
$$\max\{|f(A)+f(A)|,|A+A|\}\gg{\frac{|A|^{24/19}}{(\log|A|)^{2/19}}}.$$
\end{theorem}
\begin{theorem}\label{theorem:main3}Let $f$ be any continuous, strictly convex
 or concave function on the reals, and $A\subset\mathbb{R}$ be any finite set. Then
$$|A+f(A)|\gg{\frac{|A|^{24/19}}{(\log|A|)^{2/19}}}.$$
\end{theorem}

\subsection*{Applications to sum-product estimates}
By choosing  $f(x)=\log x$  and applying Theorems \ref{theorem:main}
and \ref{theorem:main2}, some interesting sum-product type results
can be specified, especially in the case when the productset is
small. A sum-product estimate is a bound on
$\max\{|A+A|,|A\cdot{A}|\}$, and it is conjectured that at least one
of these sets should grow to a near maximal size. Solymosi
(\cite{solymosi}) proved that
$\max\{|A+A|,|A\cdot{A}|\}\gg\frac{|A|^{4/3}}{(\log|A|)^{1/3}}$, and
this is the current best known bound. See \cite{solymosi} and the
references contained therein for more details on this problem and
its history.

In a similar spirit, one may conjecture that at least one of $|A-A|$ and $|A\cdot{A}|$ must be large,
and indeed this is somewhat true. In an earlier paper of Solymosi (\cite{solymosi2}) on sum-product estimates, it was proved that
$$\max\{|A+A|,|A\cdot{A}|\}\gg{\frac{|A|^{14/11}}{(\log|A|)^{3/11}}}.$$
It is easy to change the proof slightly in order to replace $|A+A|$
with $|A-A|$ in the above, however, in Solymosi's subsequent paper
on sum-product estimates, this was not the case. So,
$\max\{|A-A|,|A\cdot{A}|\}\gg{\frac{|A|^{14/11}}{(\log|A|)^{3/11}}}$
represents the current best known bound of this type. Applying
Theorem \ref{theorem:main} with $f(x)=\log x$, and noting that
$|f(A)+f(A)|=|A\cdot{A}|$, we get the following very marginal
improvement.
\begin{corollary}
\begin{equation}
|A\cdot{A}|^6|A-A|^5\gg{\frac{|A|^{14}}{(\log|A|)^2}}.
\label{differenceproduct}
\end{equation}
In particular, this implies that
$$\max\{|A\cdot{A}|, |A-A|\}\gg{\frac{|A|^{14/11}}{(\log|A|)^{2/11}}}.$$
\end{corollary}
By applying Theorem \ref{theorem:main2} in the same way, we establish that
\begin{equation}
|A\cdot{A}|^{10}|A+A|^9\gg{\frac{|A|^{24}}{(\log|A|)^2}}.
\label{sumproduct}
\end{equation}
In the case when the product set is small, then
(\ref{differenceproduct}) and (\ref{sumproduct}) show that the
sumset and difference set grow non-trivially. This was shown in
\cite{Liconvex}, and here we get a more explicit version of the same
result.

\end{section}

\begin{section}{Notation and Preliminaries}

Throughout this paper, the symbols $\ll,\gg$ and $\approx$ are used to suppress constants. For example, $X\ll{Y}$
means that there exists some absolute constant $C$ such that $X<CY$. $X\approx{Y}$ means that $X\ll{Y}$ and $Y\ll{X}$. Also, all logarithms are to base 2.

For sets $A$ and $B$, let $E(A,B)$ be the additive energy of $A$ and
$B$, defined in the usual way. So, defining $\delta_{A,B}(s)$ (and
respectively $\sigma_{A,B}(s)$) to be the number of representations
of an element $s$ of $A-B$ (respectively $A+B$), and
$\delta_A(s)=\delta_{A,A}(s)$, we define
$$E(A,B)=\sum_s\delta_A(s)\delta_B(s)=\sum_s\delta_{A,B}(s)^2=\sum_s\sigma_{A,B}(s)^2.$$
Given a set $A\subset{\mathbb{R}}$ and some $s\in{\mathbb{R}}$,
define $A_s:=A\cap(A+s)$. A crucial observation to make is that
$|A_s|=\delta_A(s)$. In this paper, following \cite{SS}, the third
moment energy $E_3(A)$ will also be studied, where
$$E_3(A)=\sum_s\delta_{A}(s)^3.$$
In much the same way, we define
$$E_{1.5}(A)=\sum_s\delta_{A}(s)^{1.5}.$$
Later on, we will need the following lemma, which was proven in \cite{Liconvex}.

\begin{lemma}\label{E1.5}
Let $A,B$ be any sets. Then
$$E_{1.5}(A)^2\cdot|B|^2\leq E_3(A)^{2/3}\cdot E_3(B)^{1/3}\cdot E(A,A+B).$$
\end{lemma}
\end{section}

\begin{section}{Some consequences of the Szemer\'{e}di-Trotter theorem}

The main preliminary result is an upper bound on the number of high
multiplicity elements of a sumset, a result which comes from an
application of the Szemer\'{e}di-Trotter incidence theorem
(\cite{ST}).
\begin{theorem}
Let $\mathcal{P}$ be a set of points in the plane and $\mathcal{L}$
a set of curves such that any pair of curves intersect at most once.
Then,
\[|\{(p,l)\in{\mathcal{P}\times\mathcal{L}}:p\in{l}\}|\leq{4(|\mathcal{P}||\mathcal{L}|)^{2/3}+4|\mathcal{P}|+|\mathcal{L}|}.\]
\end{theorem}

\textbf{Remark}. While this paper was in the process of being
drafted, a very similar result to the following lemma was included
in a paper of Schoen and Shkredov (\cite{SSlong}, Lemma 24) which
was posted on the arXiv. See their paper for an alternative
description of this result and proof. Note also that a weaker
version of this result was also proven in \cite{Liconvex}.

\begin{lemma}\label{lemma:main}
Let $f$ be a continuous, strictly convex
 or concave function on the reals, and $A,B,C\subset\mathbb{R}$ be finite sets
 such that $|B||C|\gg{|A|^2}$. Then for all $\tau\geq{1}$,
\begin{equation}
\label{ST1}
\big|\{x:\sigma_{f(A),C}(x)\geq\tau\}\big|\ll\frac{|A+B|^2|C|^2}{|B|\tau^3},
\end{equation}
and
\begin{equation}
\label{ST2}
\big|\{y:\sigma_{A,B}(y)\geq\tau\}\big|\ll\frac{|f(A)+C|^2|B|^2}{|C|\tau^3}.
\end{equation}

\end{lemma}

\begin{proof}

 Let $G(f)$ denote the graph of $f$ in the plane.
 For any
$(\alpha,\beta)\in\mathbb{R}^2$, put
$L_{\alpha,\beta}=G(f)+(\alpha,\beta)$. Define  a set of points
$\mathcal P=(A+B)\times(f(A)+C)$, and a set of curves $\mathcal
L=\{L_{b,c}: (b,c)\in B\times C\}$. By convexity or concavity,
$|\mathcal L|=|B||C|$, and any pair of curves from $\mathcal L$
intersect at most once.
 Let ${\mathcal P}_{\tau}$ be the set of
points of $\mathcal P$ belonging to at least $\tau$ curves from
$\mathcal L$. Applying the aforementioned Szemer\'{e}di-Trotter
theorem to ${\mathcal P}_{\tau}$ and $\mathcal L$,
\[\tau|{\mathcal P}_{\tau}|\leq 4(|{\mathcal P}_{\tau}||B||C|)^{2/3}+4|{\mathcal P}_{\tau}|+|B||C|.\]
Now we claim for any $\tau>0$ one has
\begin{equation}\label{formula 41}|{\mathcal
P}_{\tau}|\ll\frac{|B|^2|C|^2}{\tau^3}.\end{equation} The reason is
as follows. Firstly, since there is no point of $\mathcal P$
belonging to at least $\min\{|B|+1,|C|+1\}$ curves from $\mathcal
L$, to prove (\ref{formula 41}) we may assume that
$\tau\leq\sqrt{|B||C|}$. Secondly, if $\tau< 8$, then (\ref{formula
41}) holds true since
\[|{\mathcal
P}_{\tau}|\leq|{\mathcal
P}|=|(A+B)\times(f(A)+C)|\leq|A|^2|B||C|\ll|B|^2|C|^2\leq64\frac{|B|^2|C|^2}{\tau^2}.\]
Finally, we may assume that $8\leq\tau\leq\sqrt{|B||C|}$. In this
case we have \[\frac{\tau|{\mathcal P}_{\tau}|}{2}\leq4(|{\mathcal
P}_{\tau}||B||C|)^{2/3}+|B||C|.\] Thus
\[|{\mathcal
P}_{\tau}|\ll\max\{\frac{|B|^2|C|^2}{\tau^3},\frac{|B||C|}{\tau}\}=\frac{|B|^2|C|^2}{\tau^3}.\]
This proves the claim (\ref{formula 41}).

Next, suppose $\sigma_{f(A),C}(x)\geq\tau$.
 There exist
 $\tau$ distinct elements $\{a_i\}_{i=1}^{\tau}$ from $A$, $\tau$
distinct elements $\{c_i\}_{i=1}^{\tau}$ from $C$, such that
$x=f(a_i)+c_i\ (\forall i).$ Now we define $B_i\triangleq a_i+B \
(\forall i)$ and ${\mathcal
M}_x(s)\triangleq\sum_{i=1}^{\tau}\chi_{B_i}(s)$, where
$\chi_{B_i}(\cdot)$ is the characteristic function of $B_i$. Since
\[(a_i+b,x)=(a_i+b,f(a_i)+c_i)=\big(a_i,f(a_i)\big)+(b,c_i)\in L_{b,c_i}\ \ (\forall b, \forall i),\]
we have $(s,x)\in {\mathcal P}_{{\mathcal M}_x(s)}$. Note also
\[\sum_{s\in A+B}{\mathcal M}_x(s)=\sum_{i=1}^{\tau}\sum_{s\in A+B}\chi_{B_i}(s)=\tau|B|.\]
Let $M\triangleq\frac{\tau|B|}{2|A+B|}$. Then
$$\sum_{s\in A+B:\mathcal{M}_x(s)<M}\mathcal{M}_x(s)<|A+B|M=\frac{\tau|B|}{2},$$
and hence \[\sum_{s\in A+B: {\mathcal M}_x(s)\geq M }{\mathcal
M}_x(s)\geq\frac{\tau |B|}{2}.\] Dyadically decompose this sum, so
that
\begin{equation}
\sum_{j}X_j(x)\gg{\tau|B|},
\label{dyadicdyadic}
\end{equation}
where
\begin{align*}X_j(x)&\triangleq\sum_{s:M2^j\leq {\mathcal
M}_x(s)<M2^{j+1}}{\mathcal M}_x(s),\\
Y_j(x)&\triangleq\big|\big\{s\in A+B:M2^j\leq {\mathcal
M}_x(s)<M2^{j+1}\big\}\big|.\end{align*}
 By (\ref{formula 41}),
\[\sum_{x:\sigma_{f(A), C}(x)\geq\tau}Y_j(x)\leq|{\mathcal P}_{M2^j}|\ll\frac{|B|^2|C|^2}{M^32^{3j}}.\]
Noting that $ X_j(x)\approx Y_j(x)M2^j$, thus
\[\sum_{x:\sigma_{f(A), C}(x)\geq\tau}X_j(x)\ll\frac{|B|^2|C|^2}{M^22^{2j}},\]
which followed by first summing all $j$'s, then applying
(\ref{dyadicdyadic}), gives
\[\tau|B|\cdot\big|\{x:\sigma_{f(A),C}(x)\geq\tau\}\big|\ll\frac{|B|^2|C|^2}{M^2}.\]
Equivalently,
\[\big|\{x:\sigma_{f(A),C}(x)\geq\tau\}\big|\ll\frac{|A+B|^2|C|^2}{|B|\tau^3}.\]
This finishes the proof of (\ref{ST1}).

In the same way one can prove (\ref{ST2}). We only sketch the proof
as follows and leave the details to the interested readers: Suppose
$\sigma_{A,B}(y)\geq\tau$.
 There exist
 $\tau$ distinct elements $\{a_i\}_{i=1}^{\tau}$ from $A$, $\tau$
distinct elements $\{b_i\}_{i=1}^{\tau}$ from $B$, such that
$y=a_i+b_i.$ Then we define $C_i\triangleq f(a_i)+C$ and ${\mathcal
M}_y(s)\triangleq\sum_{i=1}^{\tau}\chi_{C_i}(s)$, and as before,
$(y,s)\in {\mathcal P}_{{\mathcal M}_y(s)}$. In precisely the same way as the
proof of (\ref{ST1}), one can prove that
\begin{align*}
\sum_{s\in f(A)+C: {\mathcal M}_y(s)\geq M }{\mathcal
M}_y(s)&\geq\frac{\tau |C|}{2},\\
\sum_{y:\sigma_{A,B}(y)\geq\tau}Y_j(y)\leq|{\mathcal
P}_{M2^j}|&\ll\frac{|B|^2|C|^2}{M^32^{3j}},\\
\sum_{y:\sigma_{A,B}(y)\geq\tau}X_j(y)&\ll\frac{|B|^2|C|^2}{M^22^{2j}},\\
\tau|C|\cdot\big|\{y:\sigma_{A,B}(y)\geq\tau\}\big|&\ll\frac{|B|^2|C|^2}{M^2},\\
\big|\{y:\sigma_{A,B}(y)\geq\tau\}\big|&\ll\frac{|f(A)+C|^2|B|^2}{|C|\tau^3},
\end{align*}
 where
$M\triangleq\frac{\tau|C|}{2|f(A)+C|}$,
$X_j(y)\triangleq\sum_{s:M2^j\leq {\mathcal
M}_y(s)<M2^{j+1}}{\mathcal M}_y(s)$,
$Y_j(y)\triangleq\big|\big\{s\in f(A)+C:M2^j\leq {\mathcal
M}_y(s)<M2^{j+1}\big\}\big|.$ This finishes the whole proof.
\end{proof}

\begin{corollary}\label{E3A}
Let $f$ be a continuous, strictly convex
 or concave function on the reals, and $A,C,F\subset\mathbb{R}$ be finite sets
 such that $|A|\approx|C|\ll|F|$. Then
\begin{align}
E(A,A)&\ll E_{1.5}(A)^{2/3}|f(A)+C|^{2/3}|A|^{1/3},\label{99}\\
E(A,F)&\ll |f(A)+C||F|^{3/2}, \label{1010}\\
E_3(A)&\ll|f(A)+C|^2|A|\log|A|,\label{1111}\\
E(f(A),f(A))&\ll E_{1.5}(f(A))^{2/3}|A+C|^{2/3}|A|^{1/3},\label{1212}\\
 E(f(A),F)&\ll
|A+C||F|^{3/2}, \label{1313}\\
E_3(f(A))&\ll|A+C|^2|A|\log|A|.\label{1414}
\end{align}
\end{corollary}

\begin{proof} Let $\triangle>0$ be an arbitrary real number.
First decomposing $E(A)$, then applying Lemma \ref{lemma:main} with
$B=-A$, gives
\begin{align*}
E(A,A)&=\sum_{s:\delta_A(s)<{\triangle}}\delta_A(s)^2+\sum_{j=0}^{\lfloor\log|A|\rfloor}\sum_{s:2^j\triangle\leq\delta_A(s)<2^{j+1}\triangle}\delta_A(s)^2
\\&\ll{\sqrt{\triangle}\cdot E_{1.5}(A)+\sum_{j=0}^{\lfloor\log|A|\rfloor}\frac{|f(A)+C|^2|A|}{2^{3j}\triangle^{3j}}}\cdot2^{2j}\triangle^{2j}\\
&\ll \sqrt{\triangle}\cdot
E_{1.5}(A)+\frac{|f(A)+C|^2|A|}{\triangle}\ \ \ \
\big(\triangleq\Psi(\triangle)\big).
\end{align*}
Thus $E(A)\ll\min_{\triangle>0}\Psi(\triangle)\approx
E_{1.5}(A)^{2/3}|f(A)+C|^{2/3}|A|^{1/3}$, which proves (\ref{99}).
Similarly, applying Lemma \ref{lemma:main} with $B=-F$, gives
\begin{align*}
E(A,F)&=\sum_{s:\delta_{A,F}(s)<{\triangle}}\delta_{A,F}(s)^2+\sum_{j=0}^{\lfloor\log|A|\rfloor}\sum_{s:2^j\triangle\leq\delta_{A,F}(s)<2^{j+1}\triangle}\delta_{A,F}(s)^2
\\&\ll\triangle\cdot E_{1}(A,F)+\sum_{j=0}^{\lfloor\log|A|\rfloor}\frac{|f(A)+C|^2|F|^2}{|C|2^{3j}\triangle^{3j}}\cdot2^{2j}\triangle^{2j}\\
&\ll \triangle|A||F|+\frac{|f(A)+C|^2|F|^2}{|C|\triangle}\ \ \ \
\big(\triangleq\Phi(\triangle)\big).
\end{align*}
Thus $E(A,F)\ll\min_{\triangle>0}\Phi(\triangle)\approx
|f(A)+C||F|^{1.5}$, which proves (\ref{1010}). Once again applying
Lemma \ref{lemma:main}  with $B=-A$, gives
\begin{align*}
E_3(A)&=\sum_{j=0}^{\lfloor\log|A|\rfloor}\sum_{s:2^j\leq\delta_A(s)<2^{j+1}}\delta_A(s)^3\\
&\ll\sum_{j=0}^{\lfloor\log|A|\rfloor}\frac{|f(A)+C|^2|A|}{2^{3j}\triangle^{3j}}\cdot2^{3j}\triangle^{3j}\approx|f(A)+C|^2|A|\log|A|,\end{align*}
which proves (\ref{1111}). (\ref{1212})$\sim$(\ref{1414}) can be
established by the same way. This concludes the whole proof.
\end{proof}

\end{section}

\begin{section}{Proofs of the main results}
\begin{subsection}{Proof of Theorem \ref{theorem:main}}
First, apply H\"{o}lder's inequality as follows to bound $E_{1.5}(A)$ from below:
$$|A|^6=\left(\sum_{s\in{A-A}}\delta_A(s)\right)^3\leq{\left(\sum_{s\in{A-A}}\delta_A(s)^{1.5}\right)^2|A-A|}=E_{1.5}(A)^2|A-A|.$$
Therefore, using the above bound and Lemma \ref{E1.5} with $B=-A$
gives
\[
\frac{|A|^8}{|A-A|}\leq{E_{1.5}(A)^2|A|^2}\leq{E_3(A)E(A,A-A)}.
\]
Finally, apply (\ref{1111}), and (\ref{1010}) with $F=A-A$, to
conclude that
\[\frac{|A|^8}{|A-A|}\ll{|f(A)+C|^3|A-A|^{3/2}|A|\log|A|},\] and hence
\[|f(A)+C|^6|A-A|^5\gg{\frac{|A|^{14}}{(\log|A|)^2}},\] as required.
\end{subsection}

\begin{subsection}{Proof of Theorem \ref{theorem:main2}}
Using the standard Cauchy-Schwarz bound on the additive energy, and
then (\ref{99}), we see that
\begin{align*}
\frac{|A|^{12}}{|A+A|^3}&\leq{E(A,A)^3}
\\&\ll{E_{1.5}(A)^2|f(A)+C|^2|A|}
\\&=\left(\frac{|f(A)+C|^2}{|A|}\right)E_{1.5}(A)^2|A|^2.
\end{align*}
Next, apply Lemma \ref{E1.5}, with $B=A$, to get
\[\frac{|A|^{12}}{|A+A|^3}\ll\left(\frac{|f(A)+C|^2}{|A|}\right)E_3(A)E(A,A+A),\]
and then apply (\ref{1111}), and (\ref{1010}) with $F=A+A$, to get
\[\frac{|A|^{12}}{|A+A|^3}\ll{\frac{|f(A)+C|^2}{|A|}|f(A)+C|^3|A+A|^{3/2}|A|\log|A|},\]
which, after rearranging, gives
\[|f(A)+C|^{10}|A+A|^9\gg{\frac{|A|^{24}}{(\log|A|)^2}}.\]

\end{subsection}

\begin{subsection}{Proof of Theorem \ref{theorem:main3}}
Observe that the Cauchy-Schwarz inequality applied twice tells us
that
\[\frac{|A|^{24}}{|A+f(A)|^6}\leq{E(A,f(A))^6}\leq{E(A,A)^3E(f(A),f(A))^3},\]
so that after applying (\ref{99}) and (\ref{1212}), with either
$C=A$ or $C=f(A)$,
\begin{align*}
\frac{|A|^{26}}{|A+f(A)|^6}&\leq|A|^2\cdot{E_{1.5}(A)^2|A+f(A)|^2|A|\cdot
E_{1.5}(f(A))^2|A+f(A)|^2|A|}
\\&=(E_{1.5}(A)^2|f(A)|^2)\cdot(E_{1.5}(f(A))^2|A|^2)\cdot|A+f(A)|^4
\\&\leq{E_3(A)E_3(f(A))E(A,A+f(A))E(f(A),A+f(A))|A+f(A)|^4},
\end{align*}
where the the last inequality is a consequence of two applications
of Lemma \ref{E1.5}. Next apply (\ref{1111}) and (\ref{1414}), again
with either $C=A$ or $C=f(A)$, to get
\[\frac{|A|^{26}}{|A+f(A)|^6}\leq{|A+f(A)|^8|A|^2(\log|A|)^2E(A,A+f(A))E(f(A),A+f(A))}.\]
Finally, apply (\ref{1010}) and (\ref{1313}), still with either
$C=A$ or $C=f(A)$, so that
\[\frac{|A|^{26}}{|A+f(A)|^6}\leq{|A+f(A)|^{13}|A|^2(\log|A|)^2}.\]
Then, after rearranging, we get
\[|A+f(A)|\gg{\frac{|A|^{24/19}}{(\log|A|)^{2/19}}}.\]
\end{subsection}

\end{section}

\textbf{Acknowledgements.} This first listed author was supported by
the NSF of China (11001174). The second listed author would like to thank Misha Rudnev for many helpful conversations.

\bibliographystyle{plain}
\bibliography{reviewbibliography}

\end{document}